\documentclass[12pt,twoside,reqno]{amsart}
\RequirePackage[english]{babel}
\usepackage[OT1]{fontenc}
\usepackage{type1cm}
\usepackage{subfigure}
\usepackage{enumerate}
\usepackage{amsthm}
\RequirePackage{amsmath,amsfonts,amssymb,amsthm}
\RequirePackage[dvips]{graphicx}
\usepackage{psfrag}
\usepackage{verbatim}
\usepackage[usenames]{color}
\usepackage[dvips]{graphicx}
\usepackage{epstopdf}
\numberwithin{equation}{section}
\usepackage[active]{srcltx}
\newtheorem{thm}{\textbf Theorem}[section]

\newtheorem{defi}{\textbf Definition}[section]

\newtheorem{lem}{\textbf Lemma}[section]

\newtheorem{ex}{\textbf Example}
\newtheorem{que}{\textbf Question}
\newtheorem{rem}{\textbf Remark}

\theoremstyle{remark}
\DeclareSymbolFont{bbold}{U}{bbold}{m}{n}
\DeclareSymbolFontAlphabet{\mathbbold}{bbold}

\begin{document}

\title{ Topological Rigidity of good fractal necklaces}


\author{Fan Wen}

\address{College of Mathematics\\
  School of Science and Engineering
  \\University of Tsukuba
  \\ Ibaraki 305-8575, Japan}


\email{s1936006@s.tsukuba.ac.jp}


\maketitle

\begin{abstract}
We introduce and characterize extremal $2$-cuts for good fractal necklaces. Using this characterization and the related topological properties of extremal $2$-cuts, we prove that every good necklace has a unique necklace IFS in a certain sense. Also, we prove that two good necklaces admit only rigid homeomorphisms and thus the group of self-homeomorphisms of a good necklace is countable. In addition, a certain weaker co-Hopfian property of good necklaces is also obtained.

\medskip


\medskip


\end{abstract}

\section{Introduction}

The fractal necklaces had been introduced by the author in \cite{WF}, where some conditions for fractal necklaces with no cut points are obtained. The present paper is devoted to studying the topological rigidity of good fractal necklaces. Roughly speaking, a subset of $\mathbb{R}^d$ is rigid in a certain sense if the group of its related automorphisms is small. We refer to \cite{BKM,BM} for the quasisymmetric rigidity of Schotty sets and square carpets.

A map $f:\mathbb{R}^d\to \mathbb{R}^d$ is contracting, if there exists $0<c<1$ such that $|f(x)-f(y)|\leq c|x-y|$ for all $x,y\in R^d$. Let $\{f_1, f_2, \cdots, f_n\}$ be a family of contracting maps of $\mathbb{R}^d$.
According to Hutchinson \cite{H},
there is a unique nonempty compact subset $F$ of $\mathbb{R}^d$, called the attractor of $\{f_1, f_2, \cdots, f_n\}$, such that
$F=\cup_{k=1}^nf_k(F).$

The attractor $F$ is called a \emph{fractal necklace or a necklace}, if $n\geq 3$ and $f_1, f_2, \cdots, f_n$ are contracting homeomorphisms  of $\mathbb{R}^d$ satisfying
\begin{equation}\label{necklace}
f_m(F)\cap f_k(F)=\left\{
\begin{array}{cc}
\mbox{a singleton} &\mbox{if\, $|m-k|=1$ or $n-1$}\\ \\
\emptyset &\mbox{if\, $2\leq|m-k|\leq n-2$}
\end{array}
\right.
\end{equation}
for all distinct $m,k\in \{1,2,\cdots, n\}$. In this case, the ordered family $\{f_1,f_2,\cdots,f_n\}$ is called \emph{a necklace iterated function system (NIFS)}.

Let $I=\{1,2,\cdots, n\}$. For every integer $m\geq 0$ and every sequence $i_1i_2\cdots i_m\in I^m$ write $f_{i_1i_2\cdots i_m}$ for the composition $f_{i_1}\circ f_{i_2}\circ\cdots\circ f_{i_m}$ and $F_{i_1i_2\cdots i_m}$ for $f_{i_1i_2\cdots i_m}(F)$, where we prescribe $I^0=\{\varepsilon\}$ and $f_\varepsilon=id$. We call $F_{i_1i_2\cdots i_m}$ \emph{an $m$-level copy of $F$}. Denote by $\mathcal{C}_m(F)$ the collection of $m$-level copies of $F$ and let $\mathcal{C}(F)=\cup_{m=0}^\infty\mathcal{C}_m(F)$. From now on a copy of $F$ means an $m$-level copy of $F$ for some $m\geq 0$.  By the definition, two distinct copies $A$, $B$ of $F$ have one of the following four relationships:
\begin{equation}\label{gx}
\mbox{ $A\subset B$; $B\subset A$; $A\cap B=\emptyset$; $A\cap B$ is a singleton.}
\end{equation}

For every $k\in I$ denote by $z_k$ the unique common point of the $1$-level copies $F_k$ and $F_{k+1}$. We call the ordered points $z_1,z_2,\cdots, z_n$ \emph{the main nodes of $F$}. We say that two main nodes $z_k$ and $z_m$ are \emph{adjacent}, if $|k-m|=1\mbox{ or } n-1$. For a subset $A$ of $F$ denote respectively by $\partial_FA$ and $\mbox{int}_FA$ the boundary and the interior of $A$ in the relative topology of $F$. Then we have $\partial_FF_k=\{z_{k-1},z_k\}$ for each $k\in I$ and $\sharp\,\partial_FA\geq 2$ for each $A\in\cup_{m\geq 1}\mathcal{C}_m(F)$. Hereafter denote by $\sharp$ the cardinality and prescribe
\begin{equation}\label{pre}
\mbox{$F_{n+1}=F_1$ and $z_0=z_n$}.
\end{equation}

We say that a fractal necklace $F$ with a NIFS $\{f_1,f_2,\cdots, f_n\}$ is \emph{good}, if $\sharp(F_{kj}\cap\partial_FF_k)\leq 1$ for all $k,j\in I$. In this case, we also say that the NIFS is good. Equivalently, $F$ is good if and only if $F_k$ is the smallest copy containing $\{z_{k-1}, z_k\}$ for each $k\in I$.

\begin{figure}[htbp]
{\begin{minipage}{8cm}
\centering
\includegraphics[width=8cm]{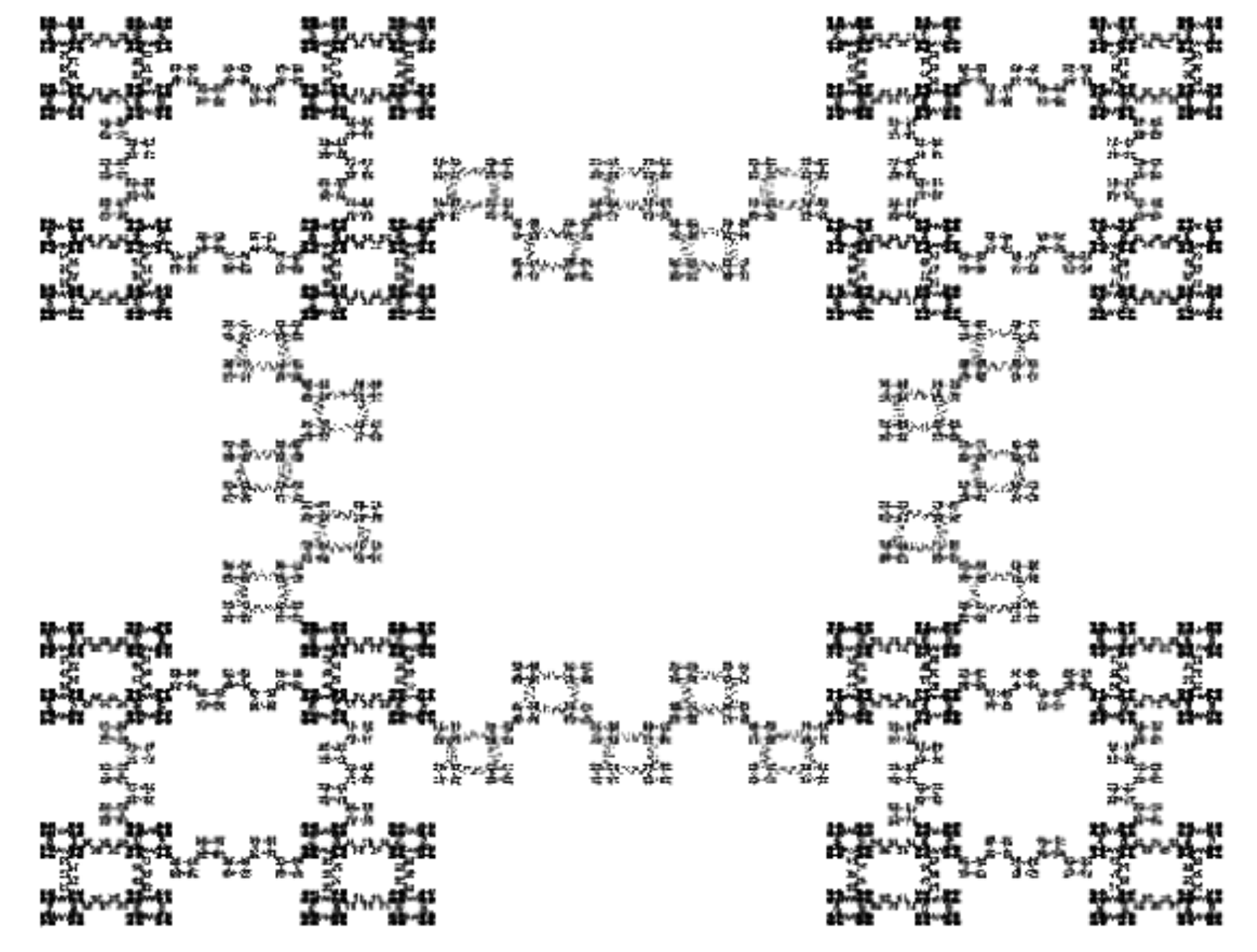}
\end{minipage}}
\caption{A necklace that is not good but has cut points.}
\end{figure}

Let $F$ be a fractal necklace. Then $F$ is path-connected and locally path-connected. Moreover, if $F$ is good then it has no cut points.
Figure 1 presents a necklace that is not good and has cut points. All of these can be found in \cite{WF}.

For a necklace, its copies and main nodes and the goodness have been defined by its given NIFS. Since two distinct NIFSs may generate the same necklace, it is natural to ask whether or not these properties of necklaces are independent of their NIFSs.

Let $\tau$ and $s$ be two permutations of $I$, where
\begin{equation}\label{necklace}
\tau(k)=\left\{
\begin{array}{cc}
k+1 &\mbox{if\, $1\leq k<n$}\\ \\
1 &\mbox{if\, $k=n$}
\end{array}
\right.
\end{equation}
and $s(k)=n-k+1$ for all $k\in I$.  Let $\mathcal{G}_n$ be the group generated by $\tau$ and $s$.
Then $\mathcal{G}_n$ is a dihedral group of $2n$ elements. Let $F$ be a necklace with a NIFS $\{f_1,f_2,\cdots,f_n\}$ on $\mathbb{R}^d$ and $\sigma\in \mathcal{G}_n$. We easily see that $\{f_{\sigma(1)},\, f_{\sigma(2)},\, \cdots, \, f_{\sigma(n)}\}$ remains a NIFS that generates the necklace $F$. We shall study the following question:

\begin{que}
What are the generating NIFSs of $F$?
\end{que}

We shall prove that every good necklace has a unique NIFS in the following sense.

\begin{thm}\label{mt1}
Let $F$ be a necklace with a good NIFS $\{f_1,f_2,\cdots,f_n\}$ on $\mathbb{R}^d$. Then for each NIFS $\{g_1,g_2,\cdots,g_m\}$ of $F$ we have

(1) $m=n$ and

(2) There is a permutation $\sigma\in \mathcal{G}_n$ such that $g_k(F)=f_{\sigma(k)}(F)$ for each $k\in I$.
\end{thm}

\begin{rem}{\rm
By Theorem \ref{mt1}, we see that, if $F$ is a necklace with a good NIFS, then all NIFSs of $F$ are good and its copies and main nodes are actually independent of the choice of its NIFSs.}
\end{rem}

\begin{defi}
We say that a homeomorphism of two necklaces $F$ and $G$ is rigid, if the image of every $m$-level copy of $F$ is an $m$-level copy of $G$  for every $m\geq 0$.
\end{defi}

Denote by $h(F,G)$ the family of homeomorphisms of $F$ onto $G$.

\begin{thm}\label{rig}
Let $F$ and $G$ be two topologically equivalent good fractal necklaces in $\mathbb{R}^d$. Then every homeomorphism of $F$ onto $G$ is rigid. Furthermore, $h(F,G)$ is countable, in particular, the group $h(F,F)$ of homeomorphisms of $F$ is countable.
\end{thm}

A topological space $X$ is  co-Hopfian, if every topological embedding of $X$ into itself is onto; see \cite{WW}. By contrast, we prove that every good fractal necklace has a weaker co-Hopfian property as follows.

\begin{thm}\label{emb}
Let $F$ and $G$ be two topologically equivalent good necklaces in $\mathbb{R}^d$ and let $h$ be a topological embedding of $F$ into $G$. Then $h(F)$ is a copy of $G$.
\end{thm}

In Theorem \ref{emb}, the assumption that $F$ and $G$ are topologically equivalent can not be removed off. Indeed, a good necklace $F$ may have a subset that is a good necklace, but it is not any copy of $F$. The readers easily see this from the standard Sierpinski gasket.

\medskip

We shall introduce and characterize extremal $2$-cuts for good necklaces in Section 2. Using this characterization and some related properties of extremal $2$-cuts, we shall prove the above theorems in Section 3. We conjecture that these theorems hold for all necklaces. However, since the extremal $2$-cuts for general necklaces are more elusive, it seems very difficult to prove (disprove) this conjecture.

\section{$2$-cuts of necklaces with no cut points}

In this section we discuss the $2$-cuts of necklaces with no cut points and the related topological invariants.

Let $X$ be a connected topological space and $A\subset X$. We say that $A$ is \emph{a cut of $X$}, if $X\setminus A$ is not connected and $X\setminus B$ is connected for each $B\subsetneq A$. A cut consisting of $k$ points is called \emph{a $k$-cut}.
A $1$-cut is also called a cut point.

For each subset $A$ of $X$ define
\begin{equation}\label{n}
N(A,X)=\sup\{\mbox{ncp}(\overline{C}): C \mbox{ is a component of } X\setminus A\},
\end{equation}
where $\overline{C}$ is the closure of $C$ in $X$, $\mbox{ncp}(\overline{C})$ denotes the number of cut points of $\overline{C}$, a component means a maximal connected subset. We say that a component $C$ of $X\setminus A$ is \emph{extremal}, if $\mbox{ncp}(\overline{C})=N(A,X)$.

For each integer $k\geq 1$ define
\begin{equation}\label{in}
N_k(X):=\sup\{N(A,X): A \mbox{ is a $k$-cut of } X\}.
\end{equation}
We say that a $k$-cut $A$ of $X$ is \emph{extremal}, if $N(A,X)=N_k(X)$.

\begin{lem}\label{inv}
Let $h:X\to Y$ be a homeomorphism of two connected topological spaces. Then we have the followings.

(1) Let $A$ be a cut of $X$. Then $h(A)$ is a cut of $Y$  and
$$N(h(A),Y)=N(A,X).$$ Moreover, $C$ is an extremal component of $X\setminus A$ if and only if $h(C)$ is an extremal component of $Y\setminus h(A)$.

(2) Let $k\geq 1$ be an integer. Then $N_k(Y)=N_k(X)$. Moreover, $A$ is an extremal $k$-cut of $X$ if and only if $h(A)$ is an extremal $k$-cut of $Y$.
\end{lem}

\begin{proof}
It is immediate.
\end{proof}

\begin{lem}\label{cs2}
Let $X$ be a connected and locally connected metric space. Let $A$ be a $k$-cut of $X$, where $k\geq 1$ is an integer. Then we have

1) $\partial_X C=A$ for every component $C$ of $X\setminus A$, and

2) $U\setminus A$ is not connected for each neighborhood $U$ of $A$ in $X$, where a subset $U$ of $X$ is called a neighborhood of $A$, if $A\subset\mbox{int}_XU$.
\end{lem}
\begin{proof} 1) Let $C$ be a component of $X\setminus A$. Then $C$ is closed in $X\setminus A$. Since $X$ is locally connected and $A$ is finite, $X\setminus A$ is locally connected, so $C$ is also open in $X\setminus A$. Thus $\partial_X C\subseteq A$.

Next we prove $A\subseteq\partial_X C$. Since $C$ is open in $X\setminus A$, we easily see that $C$ is open in $(A\setminus\partial_X C)\cup(X\setminus A)$. On the other hand, since $C$ is closed in $X\setminus A$, we have $$C=(C\cup\partial_X C)\cap(X\setminus A).$$
As $\partial_X C\subseteq A$ is proved, we have $(C\cup\partial_X C)\cap(A\setminus\partial_X C)=\emptyset$, so
$$C=(C\cup\partial_X C)\cap((A\setminus\partial_X C)\cup(X\setminus A)).$$
Thus $C$ is also closed in $(A\setminus\partial_X C)\cup(X\setminus A)$.

It then follows that $(A\setminus\partial_X C)\cup (X\setminus A)$ is not connected. Since $A$ is a cut of $X$, we get $A\setminus\partial_X C=\emptyset$,
so $A\subseteq\partial_X C$.

2) Let $U$ be a neighborhood of $A$. Suppose $U\setminus A$ is connected. Then $X\setminus A$ has a component $C$ with $C\supset U\setminus A$, so
\begin{equation}\label{clos}
A\subset\mbox{int}_X(C\cup A)
\end{equation}
As mentioned, $C$ is open in $X\setminus A$, which together with (\ref{clos}) implies that $C\cup A$ is open in $X$. On the other hand, as was shown, one has $\partial_XC=A$, which together with (\ref{clos}) implies that  $\partial_X(C\cup A)=\emptyset$, so $C\cup A$ is closed in $X$. Since $X$ is connected, we then get $C\cup A=X$, which yields $C=X\setminus A$, a contradiction.
\end{proof}

\begin{rem}{\rm
Under the condition of Lemma \ref{cs2}, if $x$ is a cut point of $X$ and $U$ is a connected neighborhood of $x$ then $x$ is a cut point of $U$.}
\end{rem}

From now on denote by $F$ a necklace with a NIFS $\{f_1, f_2,\cdots, f_n\}$ on $\mathbb{R}^d$ and by $z_1,z_2,\cdots, z_n$ its ordered main nodes.
As mentioned, we prescribe $z_0=z_n$. The main results of this section are the following theorems.

\begin{thm}\label{mp1}
If $F$ has no cut points, then
$N_2(F)=n-2$ and $\{z_{k-1},z_k\}$, $k \in I$, are extremal $2$-cuts of $F$. Moreover, if $F$ is good, then $\{z_{k-1},z_k\}$, $k \in I$, are the only extremal $2$-cuts of $F$.
\end{thm}

\begin{thm}\label{mp1+}
If $F$ has no cut points and $k\in I$, then $F\setminus F_k$ is an extremal component of $F\setminus\{z_{k-1},z_k\}$.
Moreover, if $F$ is good, then $F\setminus F_k$ is the only extremal component of $F\setminus\{z_{k-1},z_k\}$.
\end{thm}

The assumption that $F$ is good can not be removed off for the related results in Theorem \ref{mp1} and Theorem \ref{mp1+}.


\begin{ex}{\rm Let $T$ be a closed solid triangle of vertices $0, 1, v$ in the complex plane, whose corresponding angles $\alpha,\beta,\gamma$ satisfy $4\beta<2\alpha<\gamma$. Appropriately choosing a real number $a\in(0,1)$, we may construct a planar self-similar necklace $F$ by $4$ similarity maps as in Figure 2, such that its ordered main nodes $z_1, z_2, z_3, z_4$ are $0, a, a+(1-a)v,v$ respectively. This necklace is not good and has no cut points (see \cite{WF}, Theorem 2). By the first implication of Theorem \ref{mp1} we have $N_2(F)=2$. We easily check that $$\{a+a(1-a), a+a(1-a)+(1-a)^2v\}$$ is an extremal $2$-cut of $F$, but it is not equal to $\{z_{k-1},z_k\}$ for any $k\in \{1,2,3,4\}$. On the other hand, $F\setminus\{z_2,z_3\}$ has three components, two of which are extremal.}
\end{ex}

\begin{figure}[htbp]
{\begin{minipage}{7.2cm}
\centering
\includegraphics[width=7.2cm]{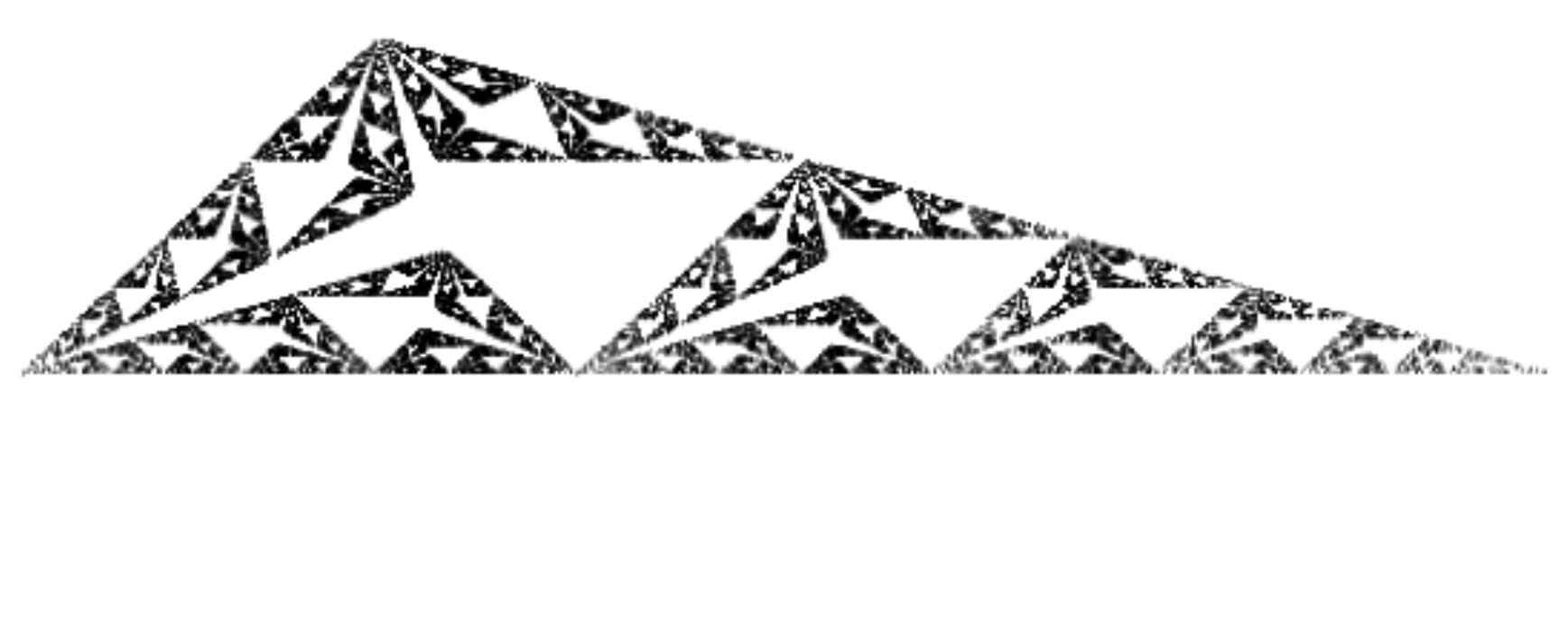}
\end{minipage}}
\caption{A necklace that is not good and has no cut points.}
\end{figure}

Let $i_1i_2\cdots i_m\in I^m$.
Since $f_1,f_2,\cdots,f_n$ have been assumed to be homeomorphisms of $\mathbb{R}^d$, we easily see that $F_{i_1i_2\cdots i_m}$ is a necklace with an induced NIFS $$\{f_{i_1\cdots i_m}\circ f_1\circ f_{i_1\cdots i_m}^{-1},\,\, f_{i_1\cdots i_m}\circ f_2\circ f_{i_1\cdots i_m}^{-1},\, \,\cdots,\,\, f_{i_1\cdots i_m}\circ f_n\circ f_{i_1\cdots i_m}^{-1}\}$$ whose main nodes are $f_{i_1\cdots i_m}(z_1),\, f_{i_1\cdots i_m}(z_2),\, \cdots,\, f_{i_1\cdots i_m}(z_n)$ and whose $1$-level copies are $F_{i_1\cdots i_m1},\, F_{i_1\cdots i_m2},\, \cdots,\, F_{i_1\cdots i_mn}$.  Let
$$M_F=\bigcup_{m=0}^\infty\,\bigcup_{i_1\cdots i_m\in I^m}\{f_{i_1\cdots i_m}(z_1),\, f_{i_1\cdots i_m}(z_2),\, \cdots,\, f_{i_1\cdots i_m}(z_n)\}.$$ Therefore $z\in M_F$ if and only if $z$ is a main node of some copy of $F$.

\medskip

\emph{From now on we assume that $F$ has no cut points.} Thus each copy of $F$ has no cut points.

\medskip

The proof of Theorems \ref{mp1} and \ref{mp1+} will occupy the rest part of this section. The connectedness and local connectedness of a necklace $F$ and the assumption that  $F$ has no cut points will be used frequently.

\begin{lem}\label{non}
Suppose $F$ has no cut points. Let $k,m\in I$, $k\neq m$. Then we have the following statements.

1) $\{z_k ,z_m\}$ is  a cut of $F$.

2) If $z_k$ and $z_m$ are not adjacent then $F\setminus\{z_{k}, z_{m}\}$ has exactly two components and $N(\{z_{k}, z_{m}\}, F)<n-2$.

3) $F\setminus F_k$ is a component of $F\setminus\{z_{k-1}, z_{k}\}$.

4) The set of cut points of $\overline{F\setminus F_k}$ is $\{z_1,z_2,\cdots, z_n\}\setminus\{z_{k-1},z_k\}.$

5) For each $z\in \{z_1,z_2,\cdots, z_n\}\setminus\{z_{k-1},z_k\}$, $\overline{F\setminus F_k}\setminus \{z\}$ has exactly two components, one containing $z_{k-1}$ and the other containing $z_k$.
\end{lem}
\begin{proof}
It is obvious.
\end{proof}

\begin{lem}\label{lcp}
Suppose $F$ has no cut points and $\{z,w\}$ is a cut of $F$. Let $A=F_{i_1i_2\cdots i_m}\in\mathcal{C}(F)$ be the smallest copy such that $\{z,w\}$ is a cut of $F_{i_1i_2\cdots i_j}$ for each $0\leq j\leq m$. Then we have

1) Both $z$ and $w$ are main nodes of $A$, and

2) $A\setminus\{z,w\}$ has exactly two components with $N(\{z,w\}, A)\leq n-2$.
\end{lem}

\begin{proof} 1) Under the assumption, since $$\lim_{m\to\infty}\max_{C\in\mathcal{C}_m(F)}\mbox{diam}(C)=0,$$ the smallest copy $A=F_{i_1i_2\cdots i_m}$ such that $\{z,w\}$ is a cut of $F_{i_1i_2\cdots i_j}$ for each $0\leq j\leq m$ does exist.

To show that $z$ and $w$ are main nodes of $A$, it suffices to prove
$$\{z,w\}\cap \mbox{int}_AB=\emptyset$$ for each $B\in \mathcal{C}_1(A)$.
In fact, suppose there is a copy $B\in \mathcal{C}_1(A)$ such that $\{z,w\}\cap \mbox{int}_AB\neq\emptyset$. Without loss of generality assume $z\in\mbox{int}_AB$. As $z$ is a cut point of $A\setminus\{w\}$, it follows from Lemma \ref{cs2} that $z$ is a cut point of $B\setminus\{w\}$, so $w\in B$ and $\{z,w\}$ is a cut of $B$, contradicting the minimality of $A$.

2) There are two cases as follows.

Case 1. $z$ and $w$ are nonadjacent main nodes of $A$. By Lemma \ref{non}, $A$ has exactly two components with
$N(\{z,w\}, A)<n-2$.

Case 2. $z$ and $w$ are adjacent main nodes of $A$. Let $B\in \mathcal{C}_1(A)$ be the copy such that $\partial_AB=\{z,w\}$. By Lemma \ref{non}, $A\setminus B$ is a component of $A\setminus\{z,w\}$ whose closure has exactly $n-2$ cut points. On the other hand, by the minimality of $A$ we see that $B\setminus\{z,w\}$ is another component of $A\setminus\{z,w\}$ whose closure has no cut points. Thus $A\setminus\{z,w\}$ has exactly two components with
$N(\{z,w\}, A)=n-2$ in this case.
\end{proof}

\begin{lem}\label{pp1}
Suppose $F$ has no cut points, $\{z,w\}$ is a cut of $F$, and $i\in I$. If $\{z,w\}$ is a cut of $F_i$ then
\begin{equation}\label{app00}
N(\{z,w\}, F)\leq \max\{n-2, N(\{z,w\}, F_{i})\}.
\end{equation}
\end{lem}

\begin{proof} There are two possible cases.

Case 1. $\{z, w\}=\{z_{i-1},z_{i}\}$. In this case, $F\setminus F_{i}$ is a component of $F\setminus\{z,w\}$ whose closure has exactly $n-2$ cut points and the other components of $F\setminus\{z,w\}$ are those of $F_{i}\setminus\{z,w\}$. Thus (\ref{app00}) holds.

Case 2. $\{z, w\}\neq\{z_{i-1},z_{i}\}$. Then, by Lemma \ref{cs2}, $F\setminus F_{i}$ is not a component of $F\setminus\{z,w\}$.
Let $A$ be the component of $F\setminus\{z,w\}$ with $$A\supset F\setminus F_{i}.$$

Subcase 1. $F_{i}\setminus\{z,w\}$ has only one component $B$ with $$B\cap\{z_{i-1},z_{i}\}\neq\emptyset.$$ In this subcase, one has $$A=(F\setminus F_{i})\cup B\mbox{ and }\{z_{i-1},z_{i}\}\subset \overline{B}.$$ Then, by the statement 5) of Lemma \ref{non}, we see that the cut points of $\overline{A}$ belong to those of $\overline{B}$, so $\mbox{ncp}(\overline{A})\leq\mbox{ncp}(\overline{B}).$

Subcase 2. $F_{i}\setminus\{z,w\}$ has two components $C$ and $D$ with $$C\cap\{z_{i-1},z_{i}\}=\{z_{i-1}\}\mbox{ and } D\cap\{z_{i-1},z_{i}\}=\{z_{i}\}.$$ In this subcase, $$A=(F\setminus F_{i})\cup C\cup D.$$ By Lemma \ref{cs2}, one has $\{z,w\}=\overline{C}\cap \overline{D}$, so $\overline{C\cup D}$ is connected, which together with the statement 5) of Lemma \ref{non} implies that the cut points of $\overline{A}$ belong to those of $\overline{C\cup D}$. As $F_i$ has no cut points, we easily see that $\overline{C\cup D}$ has no cut points, so $\mbox{ncp}(\overline{A})=0$.

Thus, for both subcases we have $\mbox{ncp}(\overline{A})\leq N(\{z,w\}, F_{i})$. As the other components of $F\setminus\{z,w\}$ belong to those of $F_{i}\setminus\{z,w\}$, we get $N(\{z,w\}, F)\leq N(\{z,w\}, F_{i})$, so (\ref{app00}) holds in Case 2.
\end{proof}

\noindent{\bf Proof of the first implication of Theorem \ref{mp1}.} Suppose $F$ has no cut points. We are going to show that $N_2(F)=n-2$ and that $\{z_{k-1} ,z_k\}$, $k\in I$, are extremal $2$-cuts of $F$.

First, we show
\begin{equation}\label{app0}
N(\{z,w\}, F)\leq n-2
\end{equation}
for each cut $\{z,w\}$ of $F$.

Let $F_{i_1i_2\cdots i_m}$ be the smallest copy such that $\{z,w\}$ is a cut of $F_{i_1i_2\cdots i_j}$ for each $0\leq j\leq m$. If $m=0$, (\ref{app0}) follows from Lemma \ref{lcp} directly. If $m\geq 1$, by Lemma \ref{pp1} we have
$$
N(\{z,w\}, F_{i_1i_2\cdots i_{j-1}})\leq \max\{n-2, N(\{z,w\}, F_{i_1i_2\cdots i_j})\}
$$
for all $1\leq j\leq m$, which together with Lemma \ref{lcp} implies
$$
N(\{z,w\}, F)\leq \max\{n-2, N(\{z,w\}, F_{i_1i_2\cdots i_m})\}\leq n-2.
$$
This proves (\ref{app0}), and thus we have $N_2(F)\leq n-2$.

Secondly, by Lemma \ref{non}, given $k\in I$, $\{z_{k-1} ,z_k\}$ is a $2$-cut of $F$ and $F\setminus F_k$ is a component of $F\setminus\{z_{k-1}, z_{k}\}$ with $\mbox{ncp}(\overline{F\setminus F_k})=n-2$, so $$N(\{z_{k-1} ,z_k\}, F)\geq  n-2,$$
which yields $N_2(F)\geq n-2$.

To sum up, we have $N_2(F)=N(\{z_{k-1} ,z_k\}, F)=n-2$  for each $k\in I$.
This completes the proof.

\medskip

Suppose $F$ has no cut points. Then $F\setminus F_k$ is connected with $$\mbox{ncp}(\overline{F\setminus F_k})=n-2$$ for each $k\in I$. However, $F\setminus F_{i_1i_2}$ may not be connected for $i_1i_2\in I^2$; see for example the necklace in Figure 1. And, in the case $F\setminus F_{i_1i_2}$ is connected, it is possible that $\mbox{ncp}(\overline{F\setminus F_{i_1i_2}})>n-2$; see for example the Sierpinski triangle. By contrast, we have the following lemma.

\begin{lem}\label{f}
Suppose $F$ is good. Then we have

1) $F\setminus A$ is connected for each $A\in\cup_{m=1}^\infty\mathcal{C}_m(F)$, and

2) $\mbox{ncp}(\overline{F\setminus A})<n-2$ for each $A\in\cup_{m=2}^\infty\mathcal{C}_m(F)$ with $\sharp\,\partial_FA=2$.
\end{lem}
\begin{proof} 1) The assumption implies that $F$ has no cut points; see \cite{WF}. So $F\setminus F_k$ is connected for each $k\in I$. Let $lj\in I^2$. Then $F\setminus F_{l}$ and $F_{l}\setminus F_{lj}$ are connected. Since $F$ is good, we may take a point $x\in (\partial_FF_{l})\setminus F_{lj}$. Then $(F\setminus F_{l})\cup\{x\}$ is connected. Observing $$((F\setminus F_{l})\cup\{x\})\cap (F_{l}\setminus F_{lj})=\{x\}$$ and $$F\setminus F_{lj}=(F\setminus F_{l})\cup\{x\}\cup(F_{l}\setminus F_{lj}),$$ we see that $F\setminus F_{lj}$ is connected. Inductively, we get that $F\setminus A$ is connected for each $A\in\cup_{m=1}^\infty\mathcal{C}_m(F)$.

2) We first prove that $\mbox{ncp}(\overline{F\setminus A})<n-2$ for each $A\in\mathcal{C}_2(F)$ with $\sharp\,\partial_FA=2$. Let such a copy $A$ be given. We may write $A=F_{lk}$ and $\partial_FF_{lk}=\{z,w\}$, where $lk\in I^2$. Then we have the following facts.

\medskip

(a) $\partial_{F_l}F_{lk}=\{z,w\}$, so $z$ and $w$ are two adjacent main nodes of $F_l$.

(b) $F_{l}\setminus F_{lk}$ is connected with $\mbox{ncp}(\overline{F_l\setminus F_{lk}})=n-2$.

(c) $\overline{F_{l}\setminus F_{lk}}=\cup_{j\in I,\,j\neq k}F_{lj}.$

(d) $F\setminus F_{lk}=(F\setminus F_{l})\cup(F_{l}\setminus F_{lk}).$

(e) $\partial_F(F\setminus F_{l})\subset\overline{F_{l}\setminus F_{lk}}.$

(f) Let $u$ be a main node of $F_l$ with $u\not\in\{z,w\}$. Then $u$ is a cut point of $\overline{F_{l}\setminus F_{lk}}$ and $\overline{F_{l}\setminus F_{lk}}\setminus\{u\}$ has exactly two components, one containing $z$ and the other containing $w$.

\medskip

We only prove (e). Since $\partial_F(F\setminus F_{l})=\partial_FF_{l}=\{z_{l-1},z_l\}$, the task is to show $\{z_{l-1},z_l\}\subset\overline{F_{l}\setminus F_{lk}}$. Suppose it is false, say $z_l\not\in\overline{F_{l}\setminus F_{lk}}$. Then one has $z_l\in F_{lk}$, so $z_l\in \partial_FF_{lk}$. By (a), $z_l$ is a main node of $F_l$, so there is a copy $F_{lj}$ such that $\{z_l\}= F_{lj}\cap F_{lk}$. Then by (c) we get $z_l\in \overline{F_{l}\setminus F_{lk}}$, a contradiction.

Now, since
$\overline{F\setminus F_{lk}}=\overline{F\setminus F_{l}}\cup\overline{F_{l}\setminus F_{lk}}$, we see from
the statement 5) of Lemma \ref{non} and (e) that the cut points of $\overline{F\setminus F_{lk}}$ belong to those of $\overline{F_{l}\setminus F_{lk}}$.
On the other hand, since $F$ is good, one has $(\partial_FF_{l})\setminus F_{lj}\neq \emptyset$ for each $j\in I$, so by (c) and (e) there are two distinct $j_1,j_2\in I\setminus\{k\}$ such that $z_{l-1}\in F_{lj_1}$ and $z_l\in F_{lj_2}$, which together with (f) implies that $\overline{F_{l}\setminus F_{lk}}$ has a cut point that is not any cut point of $\overline{F\setminus F_{lk}}$.
It then follows from (b) that $\mbox{ncp}(\overline{F\setminus F_{lk}})<n-2$.

Next let $A=F_{i_1i_2\cdots i_m}$, where $i_1i_2\cdots i_m\in I^m$, $m\geq 2$, and $\sharp\,\partial_FA=2$.
Then $F\setminus A$, $F\setminus F_{i_1}$ and $F_{i_1}\setminus A$ are connected,
$$\overline{F\setminus A}=\overline{F\setminus F_{i_1}}\cup \overline{F_{i_1}\setminus A}, \mbox{ and }\partial_F(F\setminus F_{i_1})=\partial_FF_{i_1}\subset \overline{F_{i_1}\setminus A}.$$
Thus we have by the statement 5) of Lemma \ref{non} that the cut points of $\overline{F\setminus A}$ belong to those of $\overline{F_{i_1}\setminus A}$, so
\begin{equation}\label{xgl}
\mbox{ncp}(\overline{F\setminus A})\leq\mbox{ncp}(\overline{F_{i_1}\setminus A}).
\end{equation}
Since $\sharp\,\partial_{F_{i_1\cdots i_{m-1}}}A=2$, the assumption $\sharp\,\partial_FA=2$ implies
$$
\partial_FA=\partial_{F_{i_1}}A=\cdots =\partial_{F_{i_1\cdots i_{m-1}}}A.
$$
We may repeatedly apply (\ref{xgl}) to get
$$\mbox{ncp}(\overline{F\setminus A})\leq\mbox{ncp}(\overline{F_{i_1i_2\cdots i_{m-2}}\setminus A}).$$
Since $F_{i_1i_2\cdots i_{m-2}}\setminus A$ and $F\setminus F_{i_{m-1} i_m}$ are topologically equivalent, we then get
$$
\mbox{ncp}(\overline{F\setminus A})\leq\mbox{ncp}(\overline{F\setminus F_{i_{m-1} i_m}}).
$$
As $\mbox{ncp}(\overline{F\setminus F_{i_{m-1} i_m}})<n-2$ is proved, we get $\mbox{ncp}(\overline{F\setminus A})<n-2$. \end{proof}

\noindent{\bf Proof of the second implication of Theorem \ref{mp1}.} Suppose $F$ is good. We are going to show that $\{z_{k-1} ,z_k\}$, $k\in I$, are the only extremal $2$-cuts of $F$.
As the first implication of Theorem \ref{mp1} is proved, it suffices to show
\begin{equation}\label{nn}
N(\{z,w\}, F)< n-2
\end{equation}
for each $2$-cut $\{z,w\}$ of $F$ with $\{z,w\}\not\in\{\{z_{k-1},z_k\}: k\in I\}$.

Let such a $2$-cut $\{z,w\}$ of $F$ be given. Let $A=F_{i_1i_2\cdots i_m}\in\mathcal{C}(F)$ be the smallest copy such that $\{z,w\}$ is a cut of $F_{i_1i_2\cdots i_j}$ for each $0\leq j\leq m$. By Lemma \ref{lcp}, $z$ and $w$ are main nodes of $A$ and $A\setminus\{z,w\}$ has exactly two components, which will be denoted by $B$ and $C$. Thus there is a subset $J$ of $I$ such that
\begin{equation}\label{clyx}
\overline{B}=\bigcup_{j\in J}F_{i_1i_2\cdots i_mj}\mbox{ \, and \,}\overline{C}=\bigcup_{j\in I\setminus J}F_{i_1i_2\cdots i_mj}.
\end{equation}
In addition, $F\setminus A$ is connected by 1) of Lemma \ref{f}. And we have $$F\setminus\{z,w\}=(F\setminus A)\cup B\cup C.$$

Case 1. $A=F$. By the assumption on $\{z,w\}$, we see that $z$ and $w$ are actually two nonadjacent main nodes of $F$. The inequality (\ref{nn}) follows by Lemma \ref{non}.

Case 2. $A\in\cup_{m=1}^\infty\mathcal{C}_m(F)$ and $\partial_FA=\{z,w\}$. By the assumption on $\{z,w\}$ we actually have $A\in\cup_{m=2}^\infty\mathcal{C}_m(F)$. In this case, $F\setminus A$, $B$, and $C$ are the only three components of $F\setminus\{z,w\}$. By 2) of Lemma \ref{f}, we have $$\mbox{ncp}(\overline{F\setminus A})<n-2.$$
On the other hand, since $F$ is good, we see that $z, w$ are actually two nonadjacent main nodes of $A$, which implies $$\mbox{ncp}(\overline{B})<n-2\mbox{\, and \,}\mbox{ncp}(\overline{C})<n-2.$$
Then the inequality (\ref{nn}) follows.

Case 3. $A\in\cup_{m=1}^\infty\mathcal{C}_m(F)$ and $\partial_FA\neq\{z,w\}$. Since $\partial_F(F\setminus A)=\partial_FA$, we have by Lemma \ref{cs2} that $F\setminus A$ is not a component of $F\setminus\{z,w\}$. Thus $F\setminus A$ meets exactly one of $B$ and $C$.
Without loss of generality assume that $(F\setminus A)\cap B\neq\emptyset$. Then $(F\setminus A)\cup B$ and $C$ are the only two components of $F\setminus\{z,w\}$ and we have $\partial_FA\subset \overline{B}$. Let $J$ be the subset of $I$ such that (\ref{clyx}) holds. Then $\sharp J\leq n-1$. Since $F$ is good, one also has $\sharp J\geq 2$, so
$$\mbox{ncp}(\overline{C})=\sharp(I\setminus J)-1<n-2.$$
Since $\partial_F(F\setminus A)=\partial_FA\subset \overline{B}$, arguing as we did in the proof of 2) of Lemma \ref{f}, we have that the cut points of $\overline{(F\setminus A)\cup B}$ belong to those of $\overline{B}$ and that $\overline{B}$ has at least one cut point that is not any cut point of $\overline{(F\setminus A)\cup B}$. Thus $$\mbox{ncp}(\overline{(F\setminus A)\cup B})<\mbox{ncp}(\overline{B})\leq \sharp J-1\leq n-2.$$
Then the inequality (\ref{nn}) follows. This completes the proof.

\medskip

\noindent{\bf Proof of Theorem \ref{mp1+}.} Suppose $F$ has no cut points. Fix $k\in I$. Then we have
$N_2(F)=n-2$ by the first implication of Theorem \ref{mp1}. Since $F\setminus F_k$ is a component of $F\setminus\{z_{k-1},z_k\}$ with $\mbox{ncp}(\overline{F\setminus F_k})=n-2$, it is an extremal component of $F\setminus\{z_{k-1},z_k\}$.

Now suppose $F$ is good. Let $C$ be a component of $F\setminus\{z_{k-1},z_k\}$ with $C\neq F\setminus F_k$. Then $C$ is a component of $F_k\setminus\{z_{k-1},z_k\}$. In the case $C=F_k\setminus\{z_{k-1},z_k\}$, one has $\mbox{ncp}(\overline{C})=0$. In the other case, since $F$ is good,  there is no copy $F_{kj}$ containing $\{z_{k-1},z_k\}$, so $z_{k-1}$ and $z_k$ are two nonadjacent main nodes of $F_k$, which implies $\mbox{ncp}(\overline{C})<n-2$. This proves that $F\setminus F_k$ is the only extremal component of $F\setminus\{z_{k-1},z_k\}$.

\section{The proof of main results}

\noindent{\bf Proof of Theorem \ref{mt1}.}
Let $F$ be a fractal necklace with a good NIFS $\{f_1,f_2,\cdots, f_n\}$ on $\mathbb{R}^d$. Then $F$ has no cut points. Let $\{g_1,g_2,\cdots, g_m\}$ be an arbitrary NIFS of $F$. We do not know if this NIFS is good at this stage. As the first implication of Theorem \ref{mp1} is valid for all necklaces with no cut points, we have
$$\mbox{$N_2(F)=n-2=m-2$,}$$
which yields $n=m$. We are going to show that there is a permutation $\sigma\in\mathcal{G}_n$ such that $g_k(F)=f_{\sigma(k)}(F)$ for each $k\in I$.

Let $z_1,z_2,\cdots,z_n$ be the ordered main nodes of $F$ under $\{f_1,\cdots, f_n\}$ and $w_1,w_2,\cdots,w_n$ be those of $F$ under $\{g_1,\cdots, g_n\}$.
Thus $$\{z_{k-1}, z_k\}=\partial_F(f_k(F))\mbox{\, and\, }\{w_{k-1}, w_k\}=\partial_F(g_k(F)).$$
By the first implication of Theorem \ref{mp1}, $\{w_{k-1},w_{k}\},\, k\in I,$ are extremal $2$-cuts of $F$. By the second implication of Theorem \ref{mp1}, $\{z_{k-1},z_{k}\},\, k\in I,$ are the only extremal $2$-cuts of $F$. Thus we have
\begin{equation}\label{hsy}
\{\{w_{k-1},w_{k}\}: k\in I\}=\{\{z_{k-1},z_{k}\}: k\in I\},
\end{equation}
which implies
$$\{w_1,w_2,\cdots, w_n\}=\{z_1,z_2,\cdots, z_n\} \mbox{\, as sets}.$$ Let $j\in I$ satisfy $\{w_1,w_2\}=\{z_{j-1},z_{j}\}$. Then we have
\begin{equation}\label{c2}
\mbox{$w_1=z_{j}$\, and\, $w_2=z_{j-1}$}
\end{equation}
or
\begin{equation}\label{c1}
\mbox{$w_1=z_{j-1}$\, and\, $w_2=z_{j}$}.
\end{equation}

First consider the case (\ref{c2}). Let $\sigma=\tau^{n-j}s\in\mathcal{G}_n$ be a permutation of $I$. By the definitions of $\tau$ and $s$ in Section 1 we have $\sigma(k)=j-k+1$ for each $k\in I$, hereafter we identify an integer $l$ with an integer $k\in I$ if $|l-k|=0$ or $n$. Thus (\ref{c2}) can be written as $w_1=z_{\sigma(1)}$ and $w_2=z_{\sigma(2)}$, which implies $w_k=z_{\sigma(k)}$ for each $k\in I$ by using (\ref{hsy}).

Fix $k\in I$. We then have
$$\{w_{k-1}, w_k\}=\{z_{\sigma(k-1)},z_{\sigma(k)}\}.$$
As $\sigma(k-1)=\sigma(k)+1$, we have $$\{z_{\sigma(k-1)},z_{\sigma(k)}\}=\partial_F(f_{\sigma(k-1)}(F)).$$
It follows that
$$\partial_F(g_k(F))=\partial_F(f_{\sigma(k-1)}(F)).$$
By the first implication of Theorem \ref{mp1+}, $F\setminus g_k(F)$ is an extremal component of $F\setminus\{w_{k-1}, w_k\}$. By the second implication of Theorem \ref{mp1+}, $F\setminus f_{\sigma(k-1)}(F)$ is the only extremal component of $F\setminus\{w_{k-1}, w_k\}$. Thus $$F\setminus g_k(F)=F\setminus f_{\sigma(k-1)}(F),$$ which yields $$g_k(F)=f_{\sigma(k-1)}(F)=f_{\sigma\tau^{-1}(k)}(F),$$ where $\tau^{-1}$ is the inverse of $\tau$. Of course $\sigma\tau^{-1}\in\mathcal{G}_n$.

As for the case (\ref{c1}), let $\sigma=\tau^{j-2}\in\mathcal{G}_n$ be a permutation of $I$. By a slightly easier argument, we get
$$\partial_F(g_k(F))=\partial_F(f_{\sigma(k)}(F)).$$
Using Corollary \ref{mp1+} as above, we have
$g_k(F)=f_{\sigma(k)}(F)$ for each $k\in I$. The proof is completed.

\medskip

\noindent{\bf Proof of Theorem \ref{rig}.} Let $F$ and $G$ be two topologically equivalent good necklaces in $\mathbb{R}^d$. For clarity let $\{f_1,f_2,\cdots, f_n\}$ be a NIFS of $F$ and $\{g_1,g_2,\cdots, g_m\}$ be a NIFS of $G$ on $\mathbb{R}^d$. Then, by Lemma \ref{inv} and Theorem \ref{mp1}, we have $n=m$. Denote by $z_1,z_2,\cdots,z_n$ the ordered main nodes of $F$ and by $w_1,w_2,\cdots,w_n$ those of $G$. We are going to show that every homeomorphism $h\in h(F,G)$ is rigid.

Fix $h\in h(F,G)$. Since $\{\{z_{k-1},z_k\}: k\in I\}$ is the family of extremal $2$-cuts of $F$ and $\{\{w_{k-1},w_k\}: k\in I\}$ is that of $G$ by Theorem \ref{mp1}, we have by Lemma \ref{inv}
$$\{\{h(z_{k-1}),h(z_k)\}: k\in I\}=\{\{w_{k-1},w_k\}: k\in I\}.$$
Now, arguing as we just did in the proof of Theorem \ref{mt1}, there is a permutation $\sigma\in \mathcal{G}_n$ such that $$\mbox{$h(f_k(F))=g_{\sigma(k)}(G)$}$$ for every $k\in I$. This shows that $h$ maps every $1$-level copy of $F$ onto a $1$-level copy of $G$. Inductively, one has that $h$ maps every $m$-level copy of $F$ onto an $m$-level copy of $G$ for every integer $m\geq 1$. Thus $h$ is rigid.

Next we show that $h(F,G)$ is countable. For every integer $m\geq 0$ write
$$M_{F,m}=\bigcup_{i_1i_2\cdots i_m\in I^m}f_{i_1i_2\cdots i_m}\{z_1,z_2,\cdots,z_n\}$$
and
$$M_{G,m}=\bigcup_{i_1i_2\cdots i_m\in I^m}g_{i_1i_2\cdots i_m}\{z_1,z_2,\cdots,z_n\}.$$
Then $M_{F,m}$ and $M_{G,m}$ are finite sets. Let $$M_F=\bigcup_{m=0}^\infty M_{F,m}\mbox{\, and \,}M_G=\bigcup_{m=0}^\infty M_{G,m}.$$
Then $M_F$ is dense in $F$ and $M_G$ is dense in $G$. Let $h\in h(F,G)$. As $h$ was shown to be rigid, the restriction $h|_{M_F}$ of $h$ is a bijection from  $M_F$ onto $M_G$ and satisfies $h(M_{F,m})=M_{G,m}$ for each $m\geq 0$.

Let $\Phi$ be the collection of bijections $\phi: M_F\to M_G$ satisfying
$\phi(M_{F,m})=M_{G,m}$ for each $m\geq 0$. Clearly, $\Phi$ is countable and
$$\{h|_{M_F}: h\in h(F,G)\}\subseteq \Phi,$$ so $\{h|_{M_F}: h\in h(F,G)\}$ is countable.

To prove that $h(F,G)$ is countable, it suffices to show the correspondence $h\to h|_{M_F}$ from $h(F,G)$ to $\{h|_{M_F}: h\in h(F,G)\}$
is one-to-one. In fact, let $h, \widetilde{h}\in h(F,G)$, $h\neq \widetilde{h}$, then there is a point $z\in F$ such that $h(z)\neq \widetilde{h}(z)$, so $h(A)$ and $\widetilde{h}(A)$ are disjoint for sufficiently small copy $A$ of $F$ with $z\in A$, and so $h|_{M_F}\neq \widetilde{h}|_{M_F}$ due to $M_F$ dense in $F$. This completes the proof.

\medskip

\noindent{\bf Proof of Theorem \ref{emb}.}
Let $F$ and $G$ be two topologically equivalent good fractal necklaces. Let $h:F\to G$ be a topological embedding. We are going to show that $h(F)$ is a copy of $G$.

For clarity let $\{f_1,f_2,\cdots, f_n\}$ be a NIFS of $F$ and $\{g_1,g_2,\cdots, g_m\}$ be a NIFS of $G$. Denote by $z_1,z_2,\cdots,z_n$ the main nodes of $F$ and by $w_1,w_2,\cdots,w_n$ those of $G$.

Let $A$ be the smallest copy of $G$ such that $h(F)\subseteq A$. Without loss of generality assume $A=G$. We are going to prove $h(F)=G$. Since $M_G$ is dense in $G$, it suffices to show $h(F)\supset M_G$.

By the assumption, $h(F)$ meets $\mbox{int}_GB$ for at least two $1$-level copies $B$ of $G$. Since $h(F)$ has no cut points, we see that $h(F)$ actually meets $\mbox{int}_GB$ for all $1$-level copies $B$ of $G$, so
\begin{equation}\label{li}
h(F)\supset\{w_1,w_2,\cdots, w_n\},\mbox{\, i.e. \,} h(F)\supset M_{G,0}.
\end{equation}
Next we prove $h(F)\supset M_{G,1}$. It suffices to show for each $k\in I$
\begin{equation}\label{gm}
h(F)\supset\{ g_k(w_1),g_k(w_2),\cdots, g_k(w_n)\}.
\end{equation}

By Theorem \ref{mp1}, $\{\{z_{k-1},z_k\}: k\in I\}$ is the family of extremal $2$-cuts of $F$, so $\{\{h(z_{k-1}),h(z_k)\}: k\in I\}$ is the family of extremal $2$-cuts of $h(F)$ by Lemma \ref{inv}. On the other hand, $\{w_{k-1},w_k\},\, k\in I,$ are obviously extremal $2$-cuts of $h(F)$ by (\ref{li}). It follows that
$$\{\{w_{k-1},w_k\}: k\in I\}=\{\{h(z_{k-1}),h(z_k)\}: k\in I\},$$
which implies
\begin{equation}\label{zhf}
\{h(z_1), h(z_2), \cdots, h(z_n)\}=\{w_1,w_2,\cdots, w_n\}\mbox{\, as sets}.
\end{equation}
Then, arguing as we just did in the proof of Theorem \ref{mt1}, there is a permutation $\sigma\in \mathcal{G}_n$ such that for each $k\in I$
\begin{equation}\label{zhf1}
\partial_{h(F)}(h(f_{\sigma(k)}(F)))=\partial_G(g_k(G))=\{w_{k-1}, w_k\}.
\end{equation}

Now fix $k\in I$. By (\ref{zhf1}) and the arguments of (\ref{li}), one has
$$
h(f_{\sigma(k)}(F))\subseteq g_k(G)\mbox{\,\, or \, \,} h(f_{\sigma(k)}(F))\supset\{w_1,w_2,\cdots, w_n\},
$$
in which the latter case does not occur because it contradicts (\ref{zhf}). Thus we have
$$
h(f_{\sigma(k)}(F))\subseteq g_k(G).
$$
Moreover, since $G$ is good, we see by (\ref{zhf1}) that $g_k(G)$ is the smallest copy containing $h(f_{\sigma(k)}(F))$.
Applying (\ref{li}) to the topological embedding $h: f_{\sigma(k)}(F)\to g_k(G)$, we get
$$h(f_{\sigma(k)}(F))\supset \{ g_k(w_1),g_k(w_2),\cdots, g_k(w_n)\},$$
which implies (\ref{gm}) and thus $h(F)\supset M_{G,1}$.

Inductively, we have $h(F)\supset M_{G,m}$ for every integer $m\geq 0$, and so $h(F)\supset M_{G}$. This completes the proof.

\medskip

\noindent{\bf Acknowledgement.} The author thanks Professors Fang Fuquan and Shigeki Akiyama for their helpful suggestions.

\end{document}